\documentclass[pdflatex,sn-mathphys-num]{sn-jnl}

\usepackage[utf8]{inputenc}
\usepackage{amsmath, amsthm, amssymb, amsfonts}
\usepackage{graphicx}
\usepackage{mathtools}
\usepackage{enumerate}
\usepackage{color}
\usepackage{epsfig}
\usepackage{float}
\usepackage{adjustbox} 
\usepackage{cancel}
\usepackage{epstopdf}
\usepackage{hyperref}
\usepackage{enumitem} 
\usepackage{accents}

\newcommand{\trasp}{^\top}

\newcommand{\M}{\mathcal{M}}
\newcommand{\TM}{\mathcal{T}_Y\mathcal{M}}
\newcommand{\ee}{\mathrm{e}}

\newcommand{\PhiA}{\Phi_{\frac{\tau}{2}}^A}
\newcommand{\PhiG}{\Phi_{\tau}^G}
\newcommand{\phiG}{\widetilde{\Phi}_{\tau}^G}

\newcommand{\dd}{\, \mathrm{d}}

\def\phi{\varphi}

\newtheorem{theorem}{Theorem}
\newtheorem{lemma}{Lemma}
\newtheorem{prop}{Proposition}
\newtheorem{hp}{Assumption}
\newtheorem{Remark}{Remark}

\title[Low-Rank Strang Splitting for Matrix ODEs]{Convergence of a Low-Rank Strang Splitting for Stiff Matrix Differential Equations}

\author*[1]{\fnm{Carmela} \sur{Scalone}}\email{carmela.scalone@univaq.it}
\author[2]{\fnm{Nicola} \sur{Guglielmi}}\email{nicola.guglielmi@gssi.it}

\affil*[1]{\orgdiv{Department of Engineering, Computer Science and Mathematics}, \orgname{University of L'Aquila}, \orgaddress{\city{L'Aquila}, \country{Italy}}}
\affil[2]{\orgdiv{Division of Mathematics}, \orgname{Gran Sasso Science Institute}, \orgaddress{\city{L'Aquila}, \country{Italy}}}

\abstract{
We propose and analyze a second-order Strang splitting method for a class of stiff matrix differential equations with Sylvester-type structure. The method splits the dynamics into a stiff linear part, treated exactly via matrix exponentials, and a nonlinear part, integrated by a second-order dynamical low-rank (DLR) scheme. Our main contribution is a rigorous convergence proof showing that, under suitable assumptions, the overall scheme achieves second-order accuracy. Numerical experiments confirm the theoretical results and demonstrate the robustness and efficiency of the proposed method.
}

\begin{document}

\maketitle

\section*{Introduction}

We consider the matrix differential equation
\begin{equation} \label{eq}
    \dot{X}(t) = AX(t) + X(t)A^* + G(t, X(t)), \qquad X(t_0) = X^0,
\end{equation}
where $X(t) \in \mathbb{C}^{m \times m}$, the matrix $A \in \mathbb{C}^{m \times m}$ is time-independent, and $G : [t_0, \infty) \times \mathbb{C}^{m \times m} \to \mathbb{C}^{m \times m}$ is a nonlinear operator that is assumed to be non-stiff. The structure of~\eqref{eq} arises naturally from the semi-discretization of semilinear parabolic partial differential equations, where the stiffness is due to the discretization of the elliptic operator, and the nonlinearity $G$ is typically smooth.

Equation~\eqref{eq} is often suitable to low-rank approximation. This is not only motivated by computational considerations in large-scale settings, but also by the qualitative structure of the solution itself (see~\cite{HM94}). In this context, dynamical low-rank approximation~\cite{KL} provides a powerful and flexible framework for efficiently computing low-rank solutions of matrix-valued differential equations. Several dynamical low-rank integrators have been proposed and analyzed in recent years (see~\cite{LO,CL, CKL, HMS, KLW}), including the second-order Backward Update Galerkin (BUG) scheme introduced in~\cite{BUG2}, which is of particular relevance to the present work.

In the non-stiff setting, many of these integrators perform well and offer robust convergence properties. However, such integrators show an unfavorable dependence of the error bounds on the Lipschitz constant of the right-hand side. To address this limitation, several strategies have been proposed in the literature, including implicit schemes~\cite{AC25}, exponential integrators~\cite{CV}, and splitting-based approaches~\cite{OPW}.

The Lie-Trotter-type splitting method proposed in~\cite{OPW} represents the starting point for the presented analysis. There, the stiff linear component is integrated exactly using matrix exponentials, while the nonlinear term is treated via the projector-splitting integrator. The resulting scheme is shown to be robust with respect to both stiffness and rank truncation, with convergence guarantees independent of the norm of $A$.

Building on the approach in~\cite{OPW}, we propose a second-order Strang splitting method that applies the same vector field decomposition as in~\eqref{eq}. The linear part is treated via exponential integration, while the nonlinear component is approximated using the second-order BUG integrator from~\cite{BUG2}. 
Notice that, the error bound of ~\cite{BUG2} depends on the Lipschitz constant of the right-hand side of the considered matrix ODE, for this reason, it is employed only in the solution of the non-stiff part.

The resulting scheme preserves low-rank structure and is explicit in time.

Passing from the first-order Lie-Trotter to the second-order Strang splitting poses several new challenges in the analysis. In particular, the interaction between the nonlinear low-rank integrator and the exponential treatment of the stiff terms requires careful control.
Related approaches have recently been considered in \cite{YYY}, where low-rank Strang-type methods were tested on specific model problems. In contrast, to the best of our knowledge, the present work provides the first rigorous convergence analysis for a second-order dynamical low-rank splitting method of this type, with extensive numerical validation.

The structure of the paper is as follows. In Section~2, we introduce the Strang splitting scheme and its low-rank formulation. Section~3 is devoted to the convergence analysis, under suitable regularity and compatibility assumptions. Finally, in Section~4 we present a series of numerical experiments that confirm the theoretical findings and illustrate the effectiveness of the proposed scheme.

\section{A low-rank approximation of stiff matrix differential equations}\label{procedure}

\subsection{Splitting into three subproblems}\label{split}

We aim to compute low-rank approximations to the solution of~\eqref{eq}, where the vector field naturally decomposes into a stiff linear part and a nonlinear term, see \cite{MQ, OPW}. To this end, we apply a second-order Strang splitting strategy.

We split~\eqref{eq} into three subproblems:
\begin{align}
\label{sub1}
&\dot{X}_1(t) = AX_1(t) + X_1(t)A^*, \qquad X_1(t_0) = X_1^0, \\
\label{sub2}
&\dot{X}_2(t) = G(t, X_2(t)), \qquad X_2(t_0) = X_2^0, \\
\label{sub3}
&\dot{X}_3(t) = AX_3(t) + X_3(t)A^*, \qquad X_3(t_0) = X_3^0,
\end{align}
where the linear parts~\eqref{sub1} and~\eqref{sub3} are identical up to the initial condition. We denote by $\Phi^A_\tau$ the exact solution operator of the linear subproblem over a time step $\tau$, and by $\Phi^G_\tau$ the solution operator of the nonlinear equation.
An approximate solution of \eqref{eq} is obtained by applying the Strang splitting scheme with stepsize $\tau$:
\begin{equation}
\label{strang}
\mathcal{S}_\tau := \Phi^A_{\frac{\tau}{2}} \circ \Phi^G_\tau \circ \Phi^A_{\frac{\tau}{2}}.
\end{equation}
We refer to \eqref{strang} as the full-rank Strang splitting.
An approximation $X^1$ of the solution $X(t)$ of \eqref{eq} at $t_1 =  t_0+\tau$, starting from $X^0 = X^0_1$, is given by
\[
X^1 = \mathcal{S}_\tau(X^0) =  \left( \Phi^A_{\frac{\tau}{2}} \circ \Phi^G_\tau \circ \Phi^A_{\frac{\tau}{2}}\right) \left( X^0 \right) 
\]
Observe that the numerical solution at time $t_k = t_0 + k \tau$ is $X^{k} = \mathcal{S}^k_\tau(X^0)$.

 The exact solution of the linear subproblems \eqref{sub1}, with stepsize $\tau /2$, is given in closed form as
\[
X_0^2 = \Phi^A_{\frac{\tau}{2}}(X^0) = \ee^{\frac{\tau}{2} A} \, X^0 \, \ee^{\frac{\tau}{2} A^*}
\]
and serves as initial condition for the subproblem \eqref{sub2}.
Once the solution of the subproblem \eqref{sub2} is calculated with stepsize $\tau$
\[ X_3^0 = \Phi_{\tau}^G\left(X_2^0 \right)  \]
it is used as initial condition of the subproblem \eqref{sub3}, which is solve in closed form by the exponential with stepsize $\tau/2$. After this steps we obtain the approximated solution $X^1$, which is full rank, in general.
As we will describe in the section \ref{lrsol} below, while for the linear problem we automatically have conservation of the low rank property of the solution, for the nonlinear part appropriate time integration is required.

\subsection{The low-rank integrator}\label{lrsol}
We now describe how to adapt the Strang splitting scheme to the low-rank setting. Let $r$ be a fixed target rank, and define the rank-$r$ manifold
\[
\M := \left\{ Y \in \mathbb{C}^{m \times m} : \operatorname{rank}(Y) = r \right\}.
\]
Given an initial low-rank factorization $X^0 = U_0 S_0 V_0^*$, with $U_0, V_0 \in \mathbb{C}^{m \times r}$ and $S_0 \in \mathbb{C}^{r \times r}$, we seek a low-rank approximation $X(t) \approx U(t) S(t) V(t)^*$ that evolves on the manifold $\M$.
We denote by $\TM$ the tangent space of the low-rank manifold $\M$ at a rank-$r$ matrix $Y$. 
We observe that for any $Y \in \M$, $AY+YA^* \in\TM$.
Therefore, for the linear subproblems~\eqref{sub1} and~\eqref{sub3}, the solution remains in $\M$ for all time if the initial value lies in $\M$.

Hence, no rank truncation is required and the exponential actions can be computed directly on the factors $U$ and $V$ of the low rank approximation of the solution.

For the nonlinear subproblem~\eqref{sub2}, the solution typically leaves the rank-$r$ manifold. To preserve the rank constraint, we apply the dynamical low-rank approximation framework~\cite{KL}, which projects the time derivative onto the tangent space $\mathcal{T}_Y \M$ at each time $t$. Specifically, the evolution is governed by the projected differential equation
\[
\dot{Y}(t) = P(Y(t)) G(t, Y(t)),
\]
where $P(Y)$ denotes the orthogonal projection onto $\mathcal{T}_Y \M$, which can be integrated using dedicated schemes.

In this work, we adopt the midpoint BUG integrator (BUG2) in~\cite{BUG2}, a second-order method for numerically solving projected matrix differential equations of the form~\eqref{sub2}. By combining this integrator with exact exponential propagation of the linear terms, we obtain a time integration scheme for~\eqref{eq} that operates entirely within a prescribed low-rank manifold.
Each time step of the low-rank Strang splitting method involves the following three stages:
\begin{enumerate}
    \item A half-step exponential integration of the stiff linear subproblem~\eqref{sub1}, applied to the low-rank factors;
    \item A full-step numerical integration of the nonlinear subproblem~\eqref{sub2}, projected onto the tangent space using the BUG2 scheme;
    \item A second half-step exponential integration, identical in structure to the first.
\end{enumerate}
We denote the resulting scheme as the \emph{low-rank Strang splitting method}, and define its update over one time step $\tau$ as
\begin{equation} \label{approx_strang}
	\tilde{\mathcal{S}}_{\tau} := \PhiA \circ \phiG \circ \PhiA,
\end{equation}
where $\PhiA$ denotes the exact flow associated with the stiff linear part~\eqref{sub1}, and $\phiG$ is the numerical solution operator corresponding to the dynamical low-rank approximation of the nonlinear subproblem~\eqref{sub2} via BUG2.

Given a rank-$r$ approximation $Y^0 = Y(t_0) \in \M$, where $\M$ is the manifold of rank-$r$ matrices, the approximation at time $t_1 = t_0 + \tau$ is computed as
\begin{equation} \label{method}
	Y^1 = \tilde{\mathcal{S}}_{\tau}(Y^0) = \left( \PhiA \circ \phiG \circ \PhiA \right)(Y^0).
\end{equation}
Iterating this construction, we define the sequence $\{Y^k\}_{k \geq 1}$ by
\[
Y^k = \tilde{\mathcal{S}}_{\tau}^k(Y^0), \qquad t_k = t_0 + k\tau.
\]

For completeness, we recall in the next section the BUG2 integrator from~\cite{BUG2}, used in our framework to integrate the nonlinear projected dynamics. The method is both rank-robust and second-order accurate in time. This latter property is essential to ensure that the full low-rank Strang splitting method achieves second-order convergence as a whole.



\subsection{The second order BUG integrator}
\label{sec:aug-BUG}
This section is dedicated to a brief summary of the augmented BUG integrator \cite{CKL} and the second-order midpoint BUG variant from \cite{BUG2}, which is used in our framework to integrate the nonlinear subproblem~\eqref{sub2} in low-rank form. 
\subsubsection{The augmented BUG integrator of \cite{CKL}}

We consider a matrix differential equation of the form
\[
\dot{Y}(t) = F(t, Y(t)).
\]
One time step from \( t_0 \) to \( t_1 = t_0 + \tau \), starting from a rank-\(r\) matrix \( Y_0 = U_0 S_0 V_0^* \), yields a new low-rank approximation \( Y_1 = U_1 S_1 V_1^* \). The method proceeds as follows:

\paragraph{1. Basis update (augmented bases).}  
We compute updated bases \( \widehat{U}, \widehat{V} \) of rank \( \widehat{r} = 2r \) by integrating two auxiliary matrix ODEs:

\begin{itemize}
\item \textbf{K-step:}
\[
\dot{K}(t) = F(t, K(t) V_0^*) V_0, \quad K(t_0) = U_0 S_0,
\]
followed by orthogonalization:
\[
\widehat{U} = \mathrm{orth}(U_0, K(t_1)), \qquad \widehat{M} = \widehat{U}^* U_0.
\]

\item \textbf{L-step:}
\[
\dot{L}(t) = F(t, U_0 L(t)^*)^* U_0, \quad L(t_0) = V_0 S_0^*,
\]
followed by
\[
\widehat{V} = \mathrm{orth}(V_0, L(t_1)), \qquad \widehat{N} = \widehat{V}^* V_0.
\]
\end{itemize}

\paragraph{2. Galerkin step.}  
We update the core matrix by integrating the projected equation:
\[
\dot{\widehat{S}}(t) = \widehat{U}^* F(t, \widehat{U} \widehat{S}(t) \widehat{V}^*) \widehat{V}, \quad \widehat{S}(t_0) = \widehat{M} S_0 \widehat{N}^*.
\]

The result is a rank-\(\widehat{r}\) approximation:
\[
\widehat{Y}_1 = \widehat{U} \widehat{S}(t_1) \widehat{V}^*.
\]

\paragraph{3. Truncation.}
An SVD-based truncation is applied to reduce the rank back to \(r\), either exactly or adaptively via a threshold on the singular values.

\vspace{1em}
\subsubsection{Midpoint BUG integrator (second order)}
\label{sec:midpoint-bug}

To achieve second-order accuracy in time, we apply a midpoint variant as proposed in \cite{BUG2}. Given \( Y_0 = U_0 S_0 V_0^*\), we compute an approximation \( Y_1 = U_1 S_1 V_1^* \approx Y(t_1) \) as follows:

\begin{enumerate}
\item \textbf{Midpoint prediction:}
Apply a BUG step of size \( \tau/2 \) to obtain:
\[
\widehat{Y}_{1/2} = \widehat{U}_{1/2} \widehat{S}_{1/2} \widehat{V}_{1/2}^* \approx Y(t_0 + \tau/2).
\]

\item \textbf{Galerkin update:}
Construct new bases of rank \( \overline{r} \le 4r \):
\[
\begin{aligned}
\overline{U} &= \mathrm{orth}(\widehat{U}_{1/2}, \tau F(t_{1/2}, \widehat{Y}_{1/2}) \widehat{V}_{1/2}), \\
\overline{V} &= \mathrm{orth}(\widehat{V}_{1/2}, \tau F(t_{1/2}, \widehat{Y}_{1/2})^* \widehat{U}_{1/2}),
\end{aligned}
\]
and set
\[
\overline{M} = \overline{U}^* U_0, \qquad \overline{N} = \overline{V}^* V_0.
\]

\item \textbf{Final integration:}
Integrate the core equation from \( t_0 \) to \( t_1 = t_0 + \tau \):
\[
\dot{\overline{S}}(t) = \overline{U}^* F(t, \overline{U} \overline{S}(t) \overline{V}^*) \overline{V}, \quad \overline{S}(t_0) = \overline{M} S_0 \overline{N}^*.
\]
This yields the rank-\(\overline{r}\) matrix:
\[
\overline{Y}_1 = \overline{U} \, \overline{S}(t_1) \, \overline{V}^*.
\]
\end{enumerate}

\paragraph{Truncation:}
A final SVD-based truncation reduces \( \overline{Y}_1 \) to rank \( r \) or adaptively selects \( r_1 \le \overline{r} \) such that
\begin{equation}
\label{adaptive_formula}
\left( \sum_{j = r_1 + 1}^{\overline{r}} \sigma_j^2 \right)^{1/2} \le \theta,
\end{equation}

where \( \sigma_j \) are the singular values of \( \overline{S}(t_1) \). The final approximation is then
\[
Y_1 = U_1 S_1 V_1^* \approx Y(t_1).
\]

\begin{Remark}
In \cite{BUG2}, some variants of a second order BUG method are presented. We refer to the midpoint second order BUG. Extensions to other variants are straightforward.
\end{Remark}

\section{Full convergence analysis}
 \label{conv_theorem}

\subsection{The main convergence result}

In this section, we present the analytical framework used to establish the convergence of the proposed low-rank Strang splitting method. 
Throughout the analysis, we retain the highly effective notation of \cite{OPW}.
The assumptions are formulated below, followed by the main convergence theorem. A detailed proof is provided in Section~\ref{proofs}.

Our analysis does not require the Lipschitz constants of the full right-hand side of the matrix differential equation~\eqref{eq}. This makes it well-suited to stiff problems.

We consider the Hilbert space $\mathbb{C}^{m \times m}$ equipped with the Frobenius norm $\| \cdot \|$. Let $A \in \mathbb{C}^{m \times m}$ be a fixed matrix, and let $G: [t_0, T] \times \mathbb{C}^{m \times m} \to \mathbb{C}^{m \times m}$ be a smooth nonlinear mapping. We assume that the exact solution $X(t)$ to~\eqref{eq} exists for all $t \in [t_0, T]$. 
Let $Y^0$ a rank $r$ approximation of $X^0$ satisfying
\begin{equation}
\label{delta}
\| X^0 - Y^0 \| \leq \delta,
\end{equation}
for some $\delta \geq 0$.

Let $P(Y)$ denotes the projection onto the tangent space in $Y$ to $\mathcal{M}_r$, we  assume that the rank-$r$ approximation 

\[
Y(t) = \ee^{(t - t_0)A} Y^0 \ee^{(t - t_0)A^*} + \int_{t_0}^t \ee^{(t - s)A} P(Y(s)) G(s, Y(s)) \ee^{(t - s)A^*} \,\mathrm{d}s.
\]

of the equation \eqref{eq} exists for $t\in \left[ t_0,T\right] $.

We make the following assumptions.

\begin{hp} \label{hp_split}
We assume that:
\begin{enumerate}[label=(\textit{\alph*}),ref=(\textit{\alph*})]
    \item \label{ass:A}
    There exist constants $\omega \in \mathbb{R}$ and $C_s > 0$ such that
    \begin{align}
    \label{bound}
    \| \ee^{tA} Z \ee^{tA^*} \| &\leq \ee^{\omega t} \| Z \|, \\
    \label{smooth_pr}
    \| \ee^{tA}(AZ + ZA^\top)\ee^{tA^*} \| &\leq \frac{C_s}{t} \, \ee^{\omega t} \| Z \|,
    \end{align}
    for all $t > 0$ and $Z \in \mathbb{C}^{m \times m}$. The constants $\omega$ and $C_s$ are independent of the spatial discretization parameter $m$ and, consequently, of the stiffness of the problem.
    This assumption reflects stability and smoothing properties of the exponential semigroup generated by $A$, see \cite{EN,pazy}. These are commonly satisfied when $A$ stems from the spatial discretization of a uniformly elliptic operator. In vectorized form, the bounds can be written as
\[
\| \ee^{t\mathcal{A}} z \|_2 \leq \ee^{\omega t} \| z \|_2, \qquad
\| \ee^{t\mathcal{A}} \mathcal{A} z \|_2 \leq \frac{C_s}{t} \ee^{\omega t} \| z \|_2,
\]
where $\mathcal{A} = I \otimes A + A \otimes I$ and $z = \mathrm{vec}(Z)$.

    \item \label{ass:G}
 
    The nonlinearity $G$ is  sufficiently smooth so that all derivatives appearing in the analysis are well-defined, in a neighborhood of the exact solution.
This implies local Lipschitz continuity of $G$ in a neighborhood of $X(t)$, with constants $L$ and $B$ such that:
\begin{equation}
\label{LB}
\| G(t, X_1) - G(t, X_2) \| \leq L \| X_1 - X_2 \|, \qquad \| G(t, X_1) \| \leq B,
\end{equation}
for all $X_1, X_2$ such that $\|  X_1 - X(t) \| \leq \gamma$ and in $\|  X_2 - X(t) \| \leq \gamma$, for all $t \in [t_0,T]$ and a suitable $\gamma >0$. The constants $B$ and $L$ may depend on the choice of the neighborhood. 
    \item \label{ass:MR}
    There exists $\varepsilon > 0$ such that for all $t \in [t_0, T]$, we have
    \[
    G(t, Y(t)) = M(t, Y(t)) + R(t, Y(t)),
    \]
    with $M(t, Y(t)) \in \mathcal{T}_{Y(t)} \mathcal{M}$ and $\| R(t, Y(t)) \| \leq \varepsilon$.
   
This reflects the fact that $G(t, Y(t))$ lies approximately in the tangent space $\mathcal{T}_{Y(t)} \mathcal{M}$. The remainder $R(t, Y(t))$ must be small for the low-rank approximation to be accurate and stable.

    \item \label{ass:compatibility}
    There exists a constant $C > 0$, independent by $m$, such that
    \[
    \| AG(t, X(t)) + G(t, X(t))A^*\| \leq C, \quad \forall t \in [t_0, T].
    \]
    This condition is typically satisfied when $A$ arises from the semi-discretization of an elliptic operator with homogeneous boundary conditions, and $G$ satisfies a compatibility condition at the boundary; see \cite{EO1,EO2}.
\end{enumerate}
\end{hp}

\bigskip

To facilitate the analysis of the low-rank integrator $\widetilde{\mathcal{S}}_\tau$ \eqref{approx_strang}, we decompose the global error as follows:
\begin{enumerate}[label=(\roman*)]
    \item The error of the full-rank Strang splitting:
    \[
    E^n_{\mathrm{sp}} := X(t_0 + n\tau) - \mathcal{S}_\tau^n(X^0).
    \]
    \item The propagation of the initial low-rank approximation error:
    \[
    E^n_{\delta} := \mathcal{S}_\tau^n(X^0) - \mathcal{S}_\tau^n(Y^0).
    \]
    \item The error introduced by the low-rank approximation:
    \[
    E^n_{\mathrm{lr}} := \mathcal{S}_\tau^n(Y^0) - \widetilde{\mathcal{S}}_\tau^n(Y^0).
    \]
\end{enumerate}
with $\mathcal{S}_\tau$ defined in \eqref{strang}.
The total error is then
\[
X(t_0 + n\tau) - \widetilde{\mathcal{S}}_\tau^n(Y^0) = E^n_{\mathrm{sp}} + E^n_\delta + E^n_{\mathrm{lr}}.
\]

\medskip

We now state the main convergence result of the paper, which provides a rigorous global error bound for the proposed low-rank Strang splitting integrator.

\begin{theorem}[Global error of the low-rank Strang splitting] \label{main_th}
Under Assumption~\ref{hp_split}, there exists $\tau_0 > 0$ such that, for all $0 < \tau \leq \tau_0$ and $n$ such that $t_0 + n\tau \leq T$, the global error satisfies
\[
\| X(t_0 + n\tau) - \widetilde{\mathcal{S}}_\tau^n(Y^0) \| \leq c_0 \tau^2(1 + |\log \tau|) + c_1 \delta + c_2 \tau \varepsilon,
\]
where the constants $c_0, c_1, c_2$ depend on $\omega$, $C_s$, $L$, $B$ and $T$, but are independent of $\tau$, $\varepsilon$, $\delta$ and $n$. The constants $\delta$ and $\varepsilon$ are defined in \eqref{delta} and \ref{hp_split}\ref{ass:MR}.
\end{theorem}

\begin{Remark}
The step size threshold $\tau_0$ depends only on the local Lipschitz constant $L$ and the regularity of the solution.
\end{Remark}

\subsection{Detailed convergence proofs} \label{proofs}

This section presents a complete convergence analysis of the low-rank Strang splitting scheme~\eqref{approx_strang}. The total error is decomposed into three contributions:

\begin{itemize}
    \item the error of the full-rank Strang splitting (Section~\ref{err_spl}),
    \item the error due to the low-rank approximation (Section~\ref{err_lr}),
    \item the propagation of the initial approximation error (Section~\ref{final_proof}).
\end{itemize}

Each contribution is estimated separately, and the final result in Theorem~\ref{main_th} follows by combining the bounds.

As in~\cite{OPW}, the constants appearing in the analysis, such as the Lipschitz constant $L$ and the uniform bound $B$ in Assumption~\ref{hp_split}\ref{ass:G}, depend on the size of a neighborhood around the exact solution. To control them, we ensure that all numerical approximations remain in a fixed compact neighborhood $\mathcal{U} \subset \mathbb{C}^{m \times m}$ of the exact solution over the integration interval. 
This follows recursively from the given proofs, taking into account that the arising constants can be controlled in terms of $L$, $B$ and the final time $T$. An appropriate choice of the maximum step size $\tau_0$ finally guarantees that all considered approximations stay in $\mathcal{U}$.

\subsubsection{The error of the full-rank Strang splitting} \label{err_spl}

The convergence of the full-rank splitting scheme \eqref{strang} is stated in the following theorem. 
The ideas used in the proof can be traced back to, e.g., \cite{EO1}.

\begin{prop}[Global error of the full-rank Strang splitting]\label{spl_err} 
	Under Assumption~\ref{hp_split}, the full-rank Strang splitting \eqref{strang} is second-order convergent, i.e., the error bound
	\[ \lVert X(t_0+n\tau) - \mathcal{S}_\tau^n(X^0) \rVert \leq C \tau^2 (1+\lvert \log \tau \rvert) \]
	holds uniformly on $t_0\leq t_0+n\tau \leq T$. The constant $C$ depends on $\omega$, $C_s$, $L$, $B$ and $T$, but is independent of $\tau$ and $n$. 
\end{prop}

\begin{proof}
	The solution of the matrix differential equation \eqref{eq} can be expressed by means of the variation-of-constants formula. Given the initial value $X(t_{k-1})=Z$, the solution at time $t_k = t_{k-1}+\tau$ with step size $\tau > 0$ is 
	\begin{align*}
	X(t_k) =\ee^{\tau A} Z\ee^{\tau A^*} + \int_0^\tau\ee^{(\tau-s)A} G(t_{k-1}+s,X(t_{k-1}+s))\ee^{(\tau-s)A^*}\dd s.
	\end{align*}
	The exact solution of the first full-rank subproblem \eqref{sub1} at $t_{k-1}+\tau/2$ with initial value $X_1(t_{k-1}) = Z$ is given by
	\begin{equation} \label{exact_1} 
	\Phi^A_{\frac{\tau}{2}}(Z) = X_1(t_{k-1}+\tau/2) =\ee^{\frac{\tau}{2} A} Z \ee^{\frac{\tau}{2} A^*},
	\end{equation} 
    whereas the exact solution of the second full-rank subproblem \eqref{sub2} with initial value $X_2(t_{k-1})=Z$ can be expressed as 
	\begin{equation} \label{exact_2} 
	\Phi^G_{\tau}(Z) =  Z+\tau G(t_{k-1},Z)+ \frac{\tau^2}{2} \ddot{X}_2(t_{k-1}) + \int_{0}^{\tau} (\tau-s)^2 \dddot{X}_2 (t_{k-1}+s) \dd s.
	\end{equation}
	Considering that \eqref{sub1} and \eqref{sub3} are analogous, and combing the previous expressions, we have that the full-rank Strang splitting solution    
	\begin{equation} \label{exact_full}
	\mathcal{S}_{\tau} (Z) = \ee^{\tau A} Z \ee^{\tau A^*} + \tau \ee^{\tau A} G(t_{k-1},X) \ee^{\tau A^*} + \frac{\tau^2}{2} \ee^{\frac{\tau}{2} A} \frac{\dd{G}}{\dd t}(t_{k-1}, X) \ee^{\frac{\tau}{2} A^*}+ \mathcal{O}(\tau^3)
	\end{equation}
	where $X = \ee^{\frac{\tau}{2} A} Z\ee^{\frac{\tau}{2} A^*}$.
Let $f(s)=\ee^{(\tau-s)A} G(t_{k-1}+s,X(t_{k-1}+s))\ee^{(\tau-s)A^*}$.
Exploiting the Peano form of the error of the midpoint quadrature rule, we can express

\begin{align}
&\int_0^{\tau}  \ee^{(\tau-s)A} G(t_{k-1}+s,X(t_{k-1}+s))\ee^{(\tau-s)A^*} \dd s = \\ \nonumber
& \tau \ee^{\frac{\tau}{2}A}G(t_{k-1}+\tfrac{\tau}{2},X(t_{k-1}+\tfrac{\tau}{2}))\ee^{\frac{\tau}{2}A^*}+\frac{1}{2}\int_0^\tau K(s,t) \ddot{f}(s) \dd s
\end{align}
with kernel $K(s,\tau) = s^2/4$ for $s<\tau/2$ and $K(s,\tau) = (\tau-s)^2/2$ for $s>\tau/2$. Consider that 

\begin{align}
&G(t_{k-1}+\tfrac{\tau}{2},X(t_{k-1}+\tfrac{\tau}{2}))-G(t_{k-1}, X) =\\ \nonumber
& \frac{\tau}{2} \partial_t G(t_{k-1}, X)+ \frac{\tau}{2}\partial_X G(t_{k-1}, X) \left( X(t_{k-1} +\tfrac{\tau}{2}) -X \right)+ \mathcal{O}(\tau^2)  = \\ \nonumber
&\frac{\tau}{2} \partial_t G(t_{k-1}, X)+ \frac{\tau}{2}\partial_X G(t_{k-1}, X) G(t_{k-1},X(t_{k-1})) + \\ \nonumber
& \partial_X G(t_{k-1}, X)  \int_0^{\frac{\tau}{2}} \int_0^s  \dot{\tilde{f}} \dd \xi \dd s
\end{align}
where \begin{equation}\label{ftilda}
\tilde{f}(s) = \ee^{(\frac{\tau}{2}-s)A} G(t_{k-1}+s,X(t_{k-1}+s))\ee^{(\frac{\tau}{2}-s)A^*}.
\end{equation}

So, we can write 
\begin{align*}
	X(t_k) & =\ee^{\tau A} Z\ee^{\tau A^*} + \tau \ee^{\frac{\tau}{2}A}G(t_{k-1}+\tfrac{\tau}{2},X(t_{k-1}+\tfrac{\tau}{2}))\ee^{\frac{\tau}{2}A^*}	+\frac{1}{2}\int_0^\tau K(s,t) \ddot{f}(s) \dd s
	\end{align*}
The local error at $t_k$ is given by

\begin{align*}
e^k_{sp} &= X(t_{k}) - \mathcal{S}_{\tau}(Z) \\
&= \tau^2 \ee^{\frac{\tau}{2} A} 
\left(  
    G\left(t_{k-1} + \tfrac{\tau}{2}, X(t_{k-1} + \tfrac{\tau}{2})\right) 
    - G(t_{k-1}, X) 
\right)  
\ee^{\frac{\tau}{2} A^*} \\
&\quad + \frac{\tau^2}{2} \ee^{\frac{\tau}{2} A} 
\frac{\dd G}{\dd t}(t_{k-1}, X) 
\ee^{\frac{\tau}{2} A^*}
+ \frac{1}{2} \int_0^\tau K(s,t) \ddot{f}(s) \dd s 
+ \mathcal{O}(\tau^3) \\
&= \frac{\tau^2}{2} \ee^{\frac{\tau}{2} A} 
\left( 
    \partial_t G(t_{k-1}, X) 
    + \partial_X G(t_{k-1}, X) G(t_{k-1}, X(t_{k-1})) 
\right) 
\ee^{\frac{\tau}{2} A^*} \\
&\quad + \frac{\tau}{2} \ee^{\frac{\tau}{2} A} 
\left( 
    \partial_X G(t_{k-1}, X) 
    \int_0^{\frac{\tau}{2}} \int_0^s \dot{\tilde{f}}(\xi) \dd \xi \dd s 
\right) 
\ee^{\frac{\tau}{2} A^*} \\
&\quad - \frac{\tau^2}{2} \ee^{\frac{\tau}{2} A} 
\frac{\dd G}{\dd t}(t_{k-1}, X) 
\ee^{\frac{\tau}{2} A^*}
+ \frac{1}{2} \int_0^\tau K(s,t) \ddot{f}(s) \dd s 
+ \mathcal{O}(\tau^3) \\
&= \frac{\tau^2}{2} \ee^{\frac{\tau}{2} A} 
\left( 
    \partial_X G(t_{k-1}, X) 
    \left( 
        G(t_{k-1}, X(t_{k-1})) - G(t_{k-1}, X) 
    \right) 
\right) 
\ee^{\frac{\tau}{2} A^*} \\
&\quad + \frac{\tau}{2} \ee^{\frac{\tau}{2} A} 
\left( 
    \partial_X G(t_{k-1}, X) 
    \int_0^{\frac{\tau}{2}} \int_0^s \dot{\tilde{f}}(\xi) \dd \xi \dd s 
\right) 
\ee^{\frac{\tau}{2} A^*} \\
&\quad + \frac{1}{2} \int_0^\tau K(s,t) \ddot{f}(s) \dd s 
+ \mathcal{O}(\tau^3)
\end{align*}

For conciseness, we omit explicit dependences of $G$.
 Using the commutativity of a matrix with its exponential, we have that 
the firts and the second derivatives of $\tilde{f}$ \eqref{ftilda} are given by
	\[ 
	\dot{\tilde{f}}(s) = -\ee^{(\frac{\tau}{2}-s)A} \left(AG+GA{{{^*}}} - \frac{\dd{G}}{\dd s}\right) \ee^{(\frac{\tau}{2}-s)A{{{^*}}}}
	\] 

\begin{align*}
\ddot{\tilde{f}}(s) = 
-\ee^{\left(\frac{\tau}{2} - s\right)A} 
\bigg[
& A \left(AG + GA^* \right) + \left(AG + GA^* \right)A^* \\
& - 2\left(A \frac{\dd G}{\dd s} + \frac{\dd G}{\dd s}A^*\right) 
+ \frac{\dd^2 G}{\dd s^2}
\bigg] 
\ee^{\left(\frac{\tau}{2} - s\right)A^*}
\end{align*}
%
\newpage
Exploiting the assumptions \eqref{hp_split} and recalling that the nonlinearity $G$ is assumed to be smooth in a neighborhood of the exact solution, we have the following form for the local error 
\begin{align}
\label{loc_err_formula_sp}
e^k_{sp} & =     \nonumber
-\int_0^{\tau } K(s,\tau) \ee^{\left(\tau - s\right)A} 
\bigg[
    A \left(AG + GA^* \right) 
    + \left(AG + GA^* \right)A^* \\ \nonumber
& \quad 
    - 2\left(A \frac{\dd G}{\dd s} + \frac{\dd G}{\dd s}A^* \right)
\bigg] 
\ee^{\left(\tau - s\right)A^*} \dd s 
+ O(\tau^3) \\  \nonumber
& =  
-\int_0^{\frac{\tau}{2} } \frac{s^2}{2} \ee^{\left(\tau - s\right)A} 
\bigg[
    A \left(AG + GA^* \right) 
    + \left(AG + GA^* \right)A^* \\  \nonumber
& \quad 
    - 2\left(A \frac{\dd G}{\dd s} + \frac{\dd G}{\dd s}A^* \right)
\bigg] 
\ee^{\left(\frac{\tau}{2} - s\right)A^*} \dd s \\  \nonumber
& \quad 
-\int_{\frac{\tau}{2}}^{\tau } \frac{(\tau-s)^2}{2} 
\ee^{\left(\frac{\tau}{2} - s\right)A} 
\bigg[
    A \left(AG + GA^* \right) 
    + \left(AG + GA^* \right)A^* \\  
& \quad 
    - 2\left(A \frac{\dd G}{\dd s} + \frac{\dd G}{\dd s}A^*\right)
\bigg] 
\ee^{\left(\tau - s\right)A^*} \dd s 
+ O(\tau^3).
\end{align}

	Due to the presence of the matrix $A$, we do not bound the local error \eqref{loc_err_formula_sp} directly. Instead, we solve the error recursion first.
	Recalling that $X^{n-1}~=~\mathcal{S}_{\tau}^{n-1}(X^0)$, we write the global error of the Strang splitting as 
	\begin{align*} 
	E^n_{sp} = \mathcal{S}_{\tau} (X(t_{n-1}))-\mathcal{S}_{\tau} (X^{n-1})+e_{sp}^n,
	\end{align*}
	where the first two terms represent the propagation of $E^{n-1}_{sp}$ by the numerical method $\mathcal{S}_{\tau}$, which is nonlinear.
	Making use of formula \eqref{exact_full}, we write  
	\begin{align} \label{strang_prop} 
	\mathcal{S}_{\tau} (X(t_{n-1}))-\mathcal{S}_{\tau} (X^{n-1}) = \ee^{\tau A} E^{n-1}_{sp}\ee^{\tau A^*} + \ee^{\tau A} H(X(t_{n-1}),X^{n-1})\ee^{\tau A^*},
	\end{align}
	where 
	\begin{align*} 
	H(X(t_{n-1}),X^{n-1})  = & \,\tau \left[G(t_{n-1},X(t_{n-1}))-G(t_{n-1},X^{n-1}) \right] \\
	& + \int_{0}^{\tau} (\tau-s) \left[ \ddot{X}_2(t_{n-1}+s) - \ddot{\widetilde{X}}_2 (t_{n-1}+s) \right] \dd s.
	\end{align*}
	
	The functions $X_2$ and $\widetilde{X}_2$ are the solutions of the second subproblem \eqref{sub2} with initial values $X(t_{n-1})$ and $X^{n-1}$, respectively. 
	Starting from $X(t_0)$ and $X^0$ and using  \eqref{strang_prop} for their propagation by the Strang splitting method, we rewrite the global error as
	\begin{align} \label{glob_err_formula_sp}
	\begin{split}
	E^n_{sp} = & \, \ee^{n \tau A} E^0_{sp} \ee^{n\tau A^*} + \underbrace{\sum_{k = 0}^{n-1} \ee^{(n-k)\tau A} H(X(t_{k}),X^{k})\ee^{(n-k)\tau A^*}}_{=: D_1} \\
	& + \underbrace{\sum_{k=1}^{n}\ee^{(n-k)\tau A} e^k_{sp} \ee^{(n-k)\tau A^*}}_{=: D_2}.
	\end{split}
	\end{align}
	By the choice of the initial value $X(t_0) = X^0$ we have $ \lVert E^0_{sp} \rVert = 0$. Since the expression $H$ mainly consists of the nonlinear function $G$, which by Assumption \ref{hp_split}\ref{ass:G} is Lipschitz continuous, and of its derivative $\ddot{X}_2$, which is assumed to be continuous, we have the bound 
	\begin{align*}
	\lVert H(X(t_k),X^k) \rVert \leq C (\tau \lVert E^k_{sp} \rVert + \tau^2). 
	\end{align*}
	Hence, property \eqref{bound} yields the following bound for the second term in the representation of the global error \eqref{glob_err_formula_sp}: 
	\begin{equation} \label{final_boundD1}
	\lVert D_1 \rVert \leq C \sum_{k = 0}^{n-1} \ee^{(n-k)\tau\omega} \left(\tau \lVert E^k_{sp} \rVert + \tau^2 \right) \leq C \tau  \left(\sum_{k = 0}^{n-1} \lVert E^k_{sp} \rVert + 1\right). 
	\end{equation}
	
Now, we have to bound the term $D_2$ of \eqref{glob_err_formula_sp} . 
Considering the expression of the local error $e_{sp}^k$ in formula \eqref{loc_err_formula_sp}, using  parabolic smoothing \ref{hp_split}\ref{ass:A}, the smoothness of $G$ \ref{hp_split}\ref{ass:G} and the approximation of  the n-th partial sum of the harmonic series by the logarithmic function, we start bounding the term
\begin{align*}
& \sum_{k=0}^{n-1} \left \| \ee^{(n-k)\tau A}  \left( \int_0^{\frac{\tau}{2}} {s^2} \ee^{(\tau-s)A}\left(A \frac{\dd G}{\dd s} + \frac{\dd G}{\dd s}A^*\right) \ee^{(\tau-s)A^*}  \right) \ee^{(n-k)\tau A^*}\right \| \leq \\
&  C   \sum_{k=1}^{n} \ee^{\tau \omega} \dfrac{1}{k \tau} \left( \int_0^{\frac{\tau}{2}} {s^2} \ee^{-s \omega}   {\dd s}  \right) \leq 
 C \ee^{(T-t_0) \omega} \tau^2 \log(n) \leq 
 C  \tau^2\vert \log(\tau) \vert 
\end{align*}

with $C$ independent by $\tau$ and $n$.


Using \eqref{smooth_pr} and \ref{hp_split}\ref{ass:compatibility}, we are able to bound the following term involved in $D_2$,
\begin{align*}
& \left \| \int_0^{\frac{\tau}{2} } \frac{s^2}{2} \ee^{\left(\tau - s\right)A} 
\bigg[
A \left(AG + GA^* \right) + \left(AG + GA^* \right)A^*
\bigg] 
\ee^{\left(\frac{\tau}{2} - s\right)A^*} \dd s\right\| \leq   C \int_0^{\frac{\tau}{2} }  \dfrac{s^2}{\tau-s} \dd s \leq C \tau^2
\end{align*}
The derivation of the bounds of the terms of $D_2$, involving terms of $e_{sp}^k$ with kernel $(\tau -s)^2$ is analogous. Finally,
\begin{align*}
& \sum_{k=0}^{n-1} \bigg\| 
    \ee^{(n-k)\tau A} 
    \Bigg( 
        \int_0^{\frac{\tau}{2}} \frac{s^2}{2}\, 
        \ee^{(\tau - s) A} 
        \big[
            A(AG + GA^*)  + (AG + GA^*) A^*
        \big] 
        \ee^{\left(\frac{\tau}{2} - s\right) A^*} 
        \dd s 
    \Bigg) 
    \ee^{(n-k)\tau A^*} 
\bigg\| \\
& \leq C \tau^2 \sum_{k=0}^{n-1} \ee^{(n-k)\tau \omega} \leq C \tau^2
\end{align*}

In summary, we achieve the following bound for $D_2$ in \eqref{glob_err_formula_sp}
\begin{align}\label{final_boundD2}
D_2   \leq C \tau^2(1+\log(\tau)) + e_{sp}^n
\end{align}
with $\| e_{sp}^n\| \leq C \tau^2$ directly. The constant $C$ only depends on $\omega$, $C_s$, $T$, $L$ and $B$.\\
Now, collecting the previous estimates  \eqref{final_boundD1} and \eqref{final_boundD2} yields the error bound
	\[ 
	\lVert E^n_{sp} \rVert \leq  C \tau \sum_{k = 0}^{n-1} \lVert E^k_{sp} \rVert + C \tau^2 \left( 1 +\vert \log \tau \vert  \right) . 
	\]
	The global error bound follows now from a discrete Gronwall inequality, see, e.g., \cite{Gronwall86}, 
	
	\[ 
	\lVert E^n_{sp} \rVert \leq C_2 \tau^2 \left( 1 +\vert \log \tau \vert  \right) \ee^{C_1 n \tau} = C \tau^2 \left( 1 +\vert \log \tau \vert  \right)   
	\]
	
with $C$, a suitable constant independent by $\tau$ and $n$.

\end{proof}

\subsubsection{The low-rank Strang splitting} \label{err_lr}
In this section, we compare the full-rank Strang splitting \eqref{strang} and the low-rank Strang splitting \eqref{approx_strang}.
We start with the following preliminary result. 

\begin{lemma} \label{locerr}
	Under Assumption \ref{hp_split}, the following bound holds uniformly for each $n \geq 1$ satisfying $t_0 \leq t_0 + n\tau \leq T$ 
	\begin{align*}
	\lVert  \mathcal{S}_{\tau} (Y^{n-1}) - \widetilde{\mathcal{S}}_{\tau} \rVert \leq C \ \tau (\tau^2+\varepsilon),  
	\end{align*}
	as long as $\lVert Y^{n-1}-Y(t_{n-1})\rVert \leq \gamma$ for given $\gamma > 0$, see \eqref{LB}. The constant $C$ depend on $\omega$, $L$, $B$, $T$ and the bound on the third derivatives of the exact solution, but are independent of $\tau$ and $n$.
\end{lemma}	

\begin{proof}
    For estimating  $(\PhiG - \phiG) (Y^{n-1})$ we refer to to the local error bound of  \cite{BUG2}, so that we have 
    \begin{align*}
	& \lVert\mathcal{S}_{\tau} (Y^{n-1}) - \widetilde{\mathcal{S}}_{\tau}(Y^{n-1}) \rVert  =  \lVert  (\PhiA \circ (\PhiG - \phiG) \circ \PhiA) (Y^{n-1}) \rVert  \leq \ee^{\frac{\tau}{2} \omega} \lVert (\PhiG - \phiG) (Y^{n-1}) \rVert \leq \\ & C \tau (\tau^2+\varepsilon)
	\end{align*}
	for a suitable constant $C$ independent by $\tau$ and $n$.	
\end{proof}
With this local error estimate, we can now  prove the the following results.

\begin{prop}\label{globerr}
	Under Assumption \ref{hp_split}, $E^n_{lr}= \mathcal{S}_{\tau}^n (Y^0) - \widetilde{\mathcal{S}}_{\tau}^n (Y^0)$ is uniformly bounded on $t_0 \leq t_0+n\tau \leq T$ as 
	\begin{align*}
	\lVert E^n_{lr} \rVert \leq C_1 \varepsilon \tau + C_2 \tau^2 , 
	\end{align*}
	where the constants $C_1$ and $C_2$ depend on $\omega$, $L$, $B$ and $T$, but are independent of $\tau$ and $n$.  
\end{prop}
\begin{proof}
Let $\widehat{Y}$, $\widetilde{Y} \in \M$. Employing bound \eqref{bound} and the Lip\-schitz continuity \eqref{LB} of the nonlinearity $G$, we obtain 
\begin{align*}
\begin{split}
\lVert \mathcal{S}_{\tau} (\widehat{Y}) - {\mathcal{S}}_{\tau}(\widetilde{Y}) \rVert & \leq \ee^{(L+\omega)\frac{\tau}{2}} \lVert \widehat{Y} - \widetilde{Y} \rVert,
\end{split}
\end{align*}
which shows stability of the splitting method $\mathcal{S}_{\tau}$. 
%
%

By induction:
\begin{align*}
\| E^n_{lr} \| &= \| \mathcal{S}_\tau^n(Y^0) - \tilde{\mathcal{S}}_\tau^n(Y^0) \|\leq \| \mathcal{S}_\tau^n(Y^0) - \mathcal{S}_\tau(\tilde{\mathcal{S}}_\tau^{n-1}(Y^0)) \| + \| \mathcal{S}_\tau(\tilde{\mathcal{S}}_\tau^{n-1}(Y^0)) - \tilde{\mathcal{S}}_\tau^n(Y^0) \| \leq \\
& \ee^{(L+\omega)\tau} \| E^{n-1}_{lr} \| + C \tau (\tau^2 + \varepsilon).
\end{align*}
Applying a discrete Gronwall inequality yields the stated result.
\end{proof}

\subsubsection{Proof of Theorem \ref{main_th}} \label{final_proof}
Finally, we are in a position to combine the results of the previous sections and prove the main result of this paper. 
\begin{proof}[Proof of Theorem \ref{main_th}]
What remains is to give a bound for the propagation $E^n_{\delta}$ of the initial error $\lVert X^0 - Y^0 \rVert$ by $\mathcal{S}_{\tau}$.
Due to stability of $\mathcal{S}_{\tau}$, we have the following bound
\[ \lVert \mathcal{S}_{\tau}^n (X^0) - \mathcal{S}_{\tau}^n (Y^0) \rVert \leq \ee^{(L+\omega)(T-t_0)} \lVert X^0-Y^0 \rVert. 
\] 
We can finally collect
\begin{align*}
\| X(t_0+n \tau) - {\widetilde{S}}^n_\tau \| & \leq \| E^n_{sp} \| + \| E^n_{lr} \| + \| E^n_{\delta} \| \leq  C_1 \tau^2 (1+ \vert \log{\tau} \vert ) + c_2 \varepsilon \tau + C_3 \tau^2 + c_1 \delta \leq \\
 &  c_0 \tau^2 (1+ \vert \log{\tau} \vert ) + c_1 \delta + c_2 \varepsilon \tau   
\end{align*}
where the constants comes from the results of the previous sections and, in particular, $c_0$ cointains $C_1$ and $C_3$.
\end{proof}

\subsection*{Implementation and further remarks}

To obtain a practical and fully implementable numerical method, the differential systems associated with the BUG2 integrator-presented in their continuous formulation in Section~3-must be discretized in time. 
We employ the explicit Heun method  for the numerical integration of the matrix differential equations appearing in the $K$-, $L$-, and $S$-substeps, see Section \ref{sec:aug-BUG} . This choice ensures both consistency with the global temporal order and computational efficiency, especially in the context of large-scale problems.

An important feature of the proposed low-rank Strang splitting scheme is its robustness with respect to the presence of small singular values. In particular, the linear subproblems are solved exactly through the action of matrix exponentials on low-rank matrices, which preserves the rank structure and avoids numerical instabilities. In the integration of the nonlinear subproblem, this property is inherited directly from the underlying BUG2 integrator~\cite{BUG2}.

In addition, the method can be extended to an adaptive-rank variant in a straightforward manner. This is achieved by truncating the singular values of the evolving solution at each time step according to the tolerance-based criterion~\eqref{adaptive_formula}. The resulting approximation maintains the desired low-rank structure, and the additional truncation error remains controlled by the prescribed threshold $\theta$, as analyzed in~\cite{CKL,BUG2}.

\begin{Remark}

The issue of order reduction in Strang splitting for stiff parabolic equations with nonlinear source terms is well known in the literature. It typically arises due to incompatibilities between the boundary conditions of the differential operator and the structure of the nonlinearity; see, e.g.,~\cite{EO1,EO2}.

In the present setting, such phenomena do not occur thanks to Assumption $(d)$.
\end{Remark}

\section{Numerical experiments}

This section provides numerical evidence that corroborates the theoretical convergence analysis developed in the previous sections for the proposed low-rank Strang splitting method. Our aim is to investigate its convergence behavior, rank approximation and accuracy. We consider benchmark problems drawn from differential Lyapunov equations and semilinear parabolic equations.

 To estimate the experimental convergence order, we rely on the classical Runge rule as in \cite{LO}, based on successive helved step sizes.

\subsection{Experiments with differential Lyapunov equations}

We begin our numerical study by considering the class of differential Lyapunov equations of the form
\begin{equation}
\label{dle}
\dot{X}(t) = A X(t) + X(t) A\trasp + Q, \quad X(t_0) = X_0,
\end{equation}
where \(A, Q, X(t) \in \mathbb{R}^{m \times m}\). The matrix \(A\) arises from the semi-discretization of an elliptic operator, such as the Laplacian with homogeneous Dirichlet boundary conditions. Both \(X_0\) and \(Q\) are chosen symmetric and positive semidefinite.

Under these assumptions, problem \eqref{dle} admits a unique global solution \(X(t)\) which preserves symmetry and positive semidefiniteness for all \(t \ge t_0\), see \cite{OPW} and reference therein. Moreover, since the right-hand side is constant, the Lipschitz constant of the nonlinear term \(G(X) = Q\) vanishes.
\subsubsection{Case 1: Smooth forcing term in the domain of the Laplacian}

We consider a benchmark example adapted from \cite{BUG2}, based on the heat equation with homogeneous Dirichlet boundary conditions:
\[
\frac{\partial u}{\partial t} = \Delta u + g(x, y), \qquad u|_{\partial \Omega} = 0,
\]
posed on the domain \(\Omega = [-\pi,\pi]^2\), with initial data and source term given by:
\[
u_0(x, y) = \sum_{k=1}^{10} 10^{-(k-1)} e^{-k(x^2+y^2)}, \quad
g(x, y) = \sum_{k=1}^{20} \delta_{k,1} \sin(kx) \sin(ky).
\]

The problem is discretized using second-order finite differences on a uniform \(m \times m\) grid with \(m = 500\), yielding the matrix differential equation:
\begin{equation}
\label{heat}
\dot{X}(t) = A X(t) + X(t) A^\top + G, \qquad X(0) = U_0 S_0 U_0^\top,
\end{equation}
where $A$ is the discretized Laplacian 
and \(G\) is a time-independent matrix constructed from the discretized source term. 
The final time is $T = 0.3$.
The compatibility between the source term and the boundary conditions ensures that the second-order convergence of the Strang splitting is preserved. The right-hand side being constant further simplifies the integration of each substep.

\begin{figure}[h]
    \centering
    \includegraphics[width=0.5\linewidth]{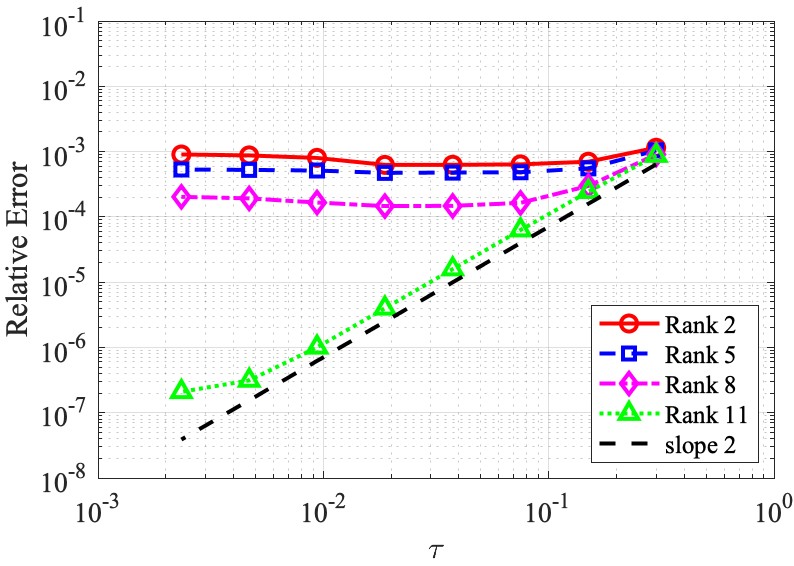}
   
    \caption{Error of the low-rank Strang splitting as a function of the time step \(\tau\) and the approximation rank \(r\), for problem \eqref{heat}. Second-order convergence is observed when the rank is sufficiently high.}
\end{figure}

\subsubsection{Case 2: Order reduction due to incompatible forcing term}

We now consider the differential Lyapunov equation \eqref{dle} with more general data. The initial condition \(X_0\) is a random matrix of rank 10, and the inhomogeneity \(Q\) is a random matrix of rank 5, generated independently of the discretization matrix \(A\).

In this setting, the data do not necessarily satisfy the compatibility assumptions required to ensure regularity of the solution. In particular,  $Q$ is not in the domain of the Laplacian, which implies that the quantity \(A Q\) may grow with the dimension and cannot be uniformly bounded. As a result, the solution lacks the smoothness needed to observe the expected second-order convergence, see also \cite{EO1,EO2,MOPP}.

This is reflected in the following experiment, where we estimate the order of convergence using the Runge rule.

\begin{table}[h!]
    \centering
    \begin{adjustbox}{width=\textwidth}
    \begin{tabular}{|c|c|c|c|c|c|c|c|c|c|c|c|}
        \hline
        $\tau$ & 0.1000 & 0.0500 & 0.0250 & 0.0125 & 0.0063 & 0.0031 & 0.0016 & 0.0008 & 0.0004 & 0.0002 & 0.0001 \\
        \hline
        $p$    & --     & --     & --     & 5.73 & 0.94 & 0.85 & 1.11 & 0.90 & 0.68 & 0.96 & 1.21 \\
        \hline
    \end{tabular}
    \end{adjustbox}
    \caption{Estimated convergence orders for the Lyapunov equation with random forcing term using the Runge rule, for fixed rank $11$, $m=128$ and $T=0.3$. }
    \label{tab_lya}
\end{table}

\begin{figure}[h]
    \centering
   \includegraphics[width=0.5\linewidth]{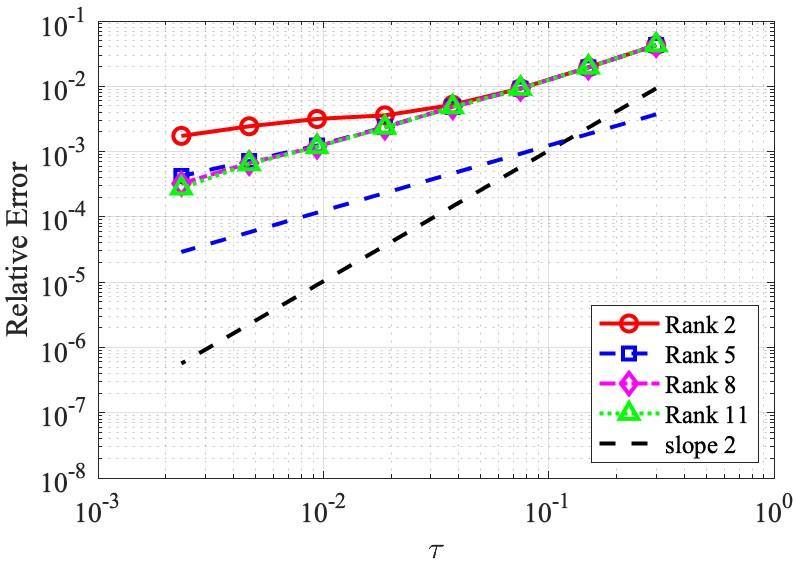}
    
    \caption{Error of the low-rank Strang splitting as a function of the time step and approximation rank for the Lyapunov equation with random low-rank forcing term. The convergence order deteriorates due to the lack of compatibility between the inhomogeneity and the Laplacian.}
\end{figure}

\subsection{Semilinear parabolic problem}

We consider the semilinear parabolic partial differential equation
\[
\partial_t v = \alpha \Delta v + v^3, \qquad v(0,x,y) = 16\,x(1-x)y(1-y),
\]
posed on the spatial domain $\Omega = [0,1]^2$ with homogeneous Dirichlet boundary conditions. The diffusion coefficient is set to $\alpha = 1/50$. The cubic nonlinearity  acts pointwise, see \cite{OPW}.

Spatial discretization is performed using second-order central finite differences on a uniform grid with $m = 500$ interior points per spatial direction. This yields a matrix differential equation of the form
\begin{equation}
\label{cubica}
\dot{U} = \alpha \left(AU+U A^\top\right)  + U^{\circ 3}, \qquad U(0) = U_0,
\end{equation}
where $U = U(t) \in \mathbb{R}^{m \times m}$, $A$ is the discretized Laplacian, and $U^{\circ 3}$ denotes the elementwise (Hadamard) cube of $U$.

To generate a reference solution, we employ the full-rank Strang splitting method with a very fine time step $\tau = 10^{-6}$. Figure~\ref{singval} displays the decay of the first ten singular values of $U(t)$ at final time $T = 0.3$. The fast decay suggests that low-rank approximation is appropriate for this problem.

\begin{figure}[h]
    \centering
   \includegraphics[width=0.5\linewidth]{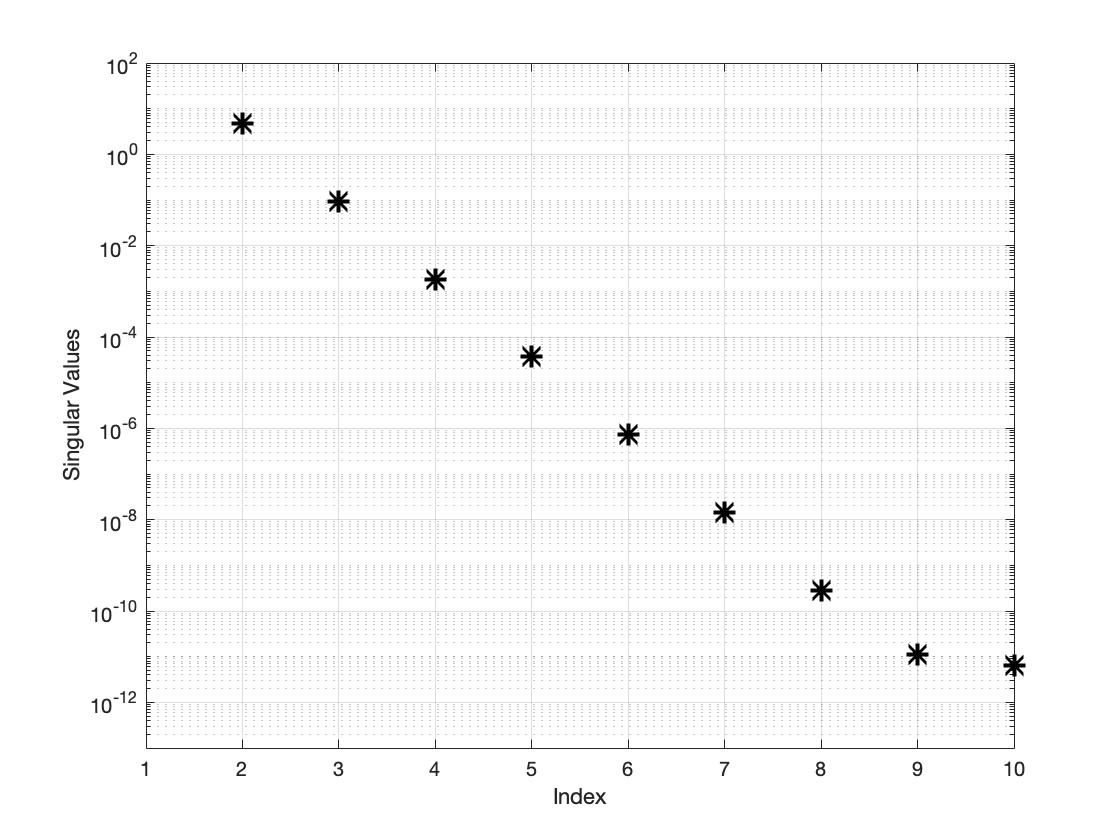}
  \caption{Decay of the first 10 singular values of the numerical reference solution, computed by the full-rank Strang splitting and time step $\tau = 10^{-6}$, at time $T=0.3$.}
    \label{singval}
\end{figure}

In Figure~\ref{errcubica}, we report the Frobenius norm error of the low-rank Strang splitting approximation with varying approximation ranks and time steps. For ranks 4 and 5, the scheme attains second-order accuracy, consistent with the theory. For lower ranks, the error stagnates, indicating that the rank is insufficient to capture the solution dynamics.
when the
approximation rank is too low, the outer error becomes independent of any step
size refinement.

\begin{figure}[h]
    \centering
    \includegraphics[width=0.5\linewidth]{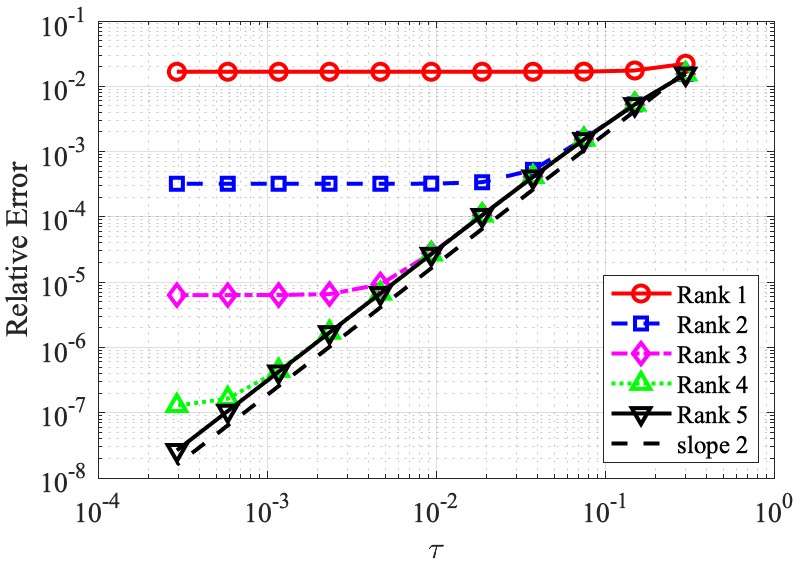}

    \caption{Frobenius norm error of the low-rank Strang splitting method as a function of time step and rank for problem~\eqref{cubica}. Second-order convergence is observed for ranks $\geq 4$.}
    \label{errcubica}
\end{figure}

To further assess the flexibility of the scheme, we test its rank-adaptive version on the same problem. Starting with rank 3, the algorithm dynamically increases the rank using the truncation criterion described in~\cite{CKL}, with threshold $\theta = 10^{-8}$. The integration is carried out with a fixed time step $\tau = 0.005$. The results are shown in Figure~\ref{adaptive}.

\begin{figure}[h]
    \centering
    \includegraphics[width=0.5\linewidth]{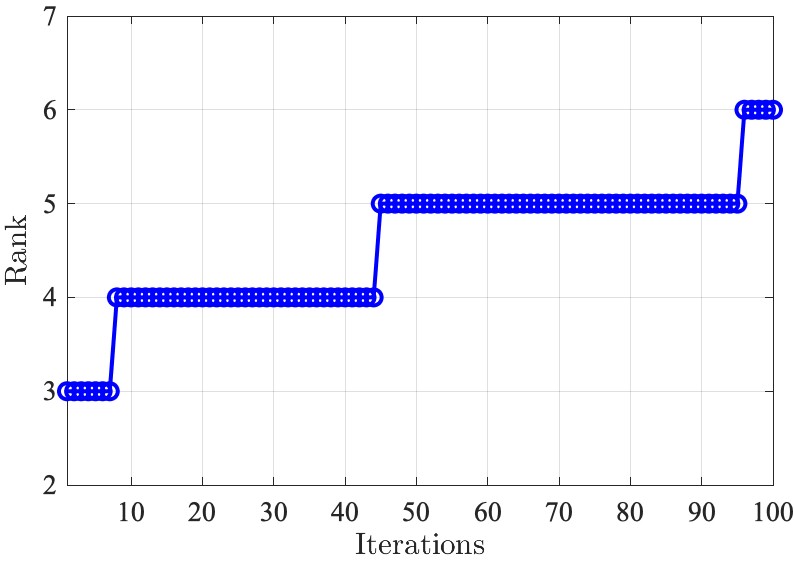}

    \caption{Rank-adaptive integration for problem~\eqref{cubica}, starting from rank 3. Time step $\tau = 0.005$, truncation threshold $\theta = 10^{-8}$.}
    \label{adaptive}
\end{figure}

\section*{Conclusion}
In this paper, we have proposed a low-rank Strang splitting method for matrix differential equations of the form \eqref{eq}. We have explicitly stated the assumptions on the problem that guarantee second-order convergence and provided a thorough analysis of the various error components. Our analysis clarifies the behavior of the method in all relevant scenarios, including cases where order reduction occurs. This comprehensive understanding not only justifies the use of Strang splitting in the existing literature but also provides a foundation for alternative constructions and extensions of the method, potentially guiding future developments in low-rank time integration techniques.

\section*{Acknowledgments}
 The authors are members of the INdAM-GNCS (Gruppo Nazionale di Calcolo Scientifico).
N.G. acknowledges that his research was supported by funds from the Italian MUR (Ministero dell'Universit\`a e della Ricerca) within the PRIN 2022 Project ``Advanced numerical methods for time dependent parametric partial differential equations with applications''  and PRIN-PNRR grant FIN4GEO. 
C. S. is partially supported by PRIN 2022 projects  (20229P2HEA), CUP: E53C24002280006, by PRIN PNRR  2022 - Prog. P20228C2PP- CUP: E53D23017940001 and by the European Union - NextGenerationEU under the Italian Ministry of University and Research (MUR) National National Centre for HPC, Big Data and Quantum Computing (CN-00000013 -  CUP: E13C22001000006).

\end{document}